\documentclass{article}

\usepackage[english]{babel}

\usepackage[letterpaper,top=2cm,bottom=2cm,left=3cm,right=3cm,marginparwidth=1.75cm]{geometry}

\usepackage{amsmath}
\usepackage{amsthm}
\newtheorem{definition}{Definition}
\newtheorem{remark}{Remark}
\newtheorem{corollary}{Corollary}
\usepackage{graphicx}
\usepackage[colorlinks=true, allcolors=blue]{hyperref}
\usepackage{float}

\title{A Construction of Rational Seifert Surface in Lens Space}
\author{Han Zhang\footnote{PKU; 1901110026@pku.edu.cn}
}

\begin{document}

\maketitle

\begin{abstract}
In this note, we give a method to construct rational Seifert surface for those smooth or piece-wise linear oriented knots in Lens space $L(p,q)$. We assume that the oriented knot has a regular projection on Heegaard torus and then construct rational Seifert surface on twist toroidal diagram.
\end{abstract}

\section{Introduction}

~~~~The existence of Seifert surface of a null-homologous knot or link is a very interesting problem in topology. In chapter.5.A.4\cite{RF}, Rolfsen showed us a direct way to constructing Seifert surface by regular projection of a smooth or piece-wise linear knot. It's a natural question whether we can generalize Seifert surface of a link. In section 1 of\cite{KJ} ,Kenneth Baker and John Etnyre defined rational Seifert suface for a knot which represents a torsion element in homology group $H_1$. Especially, $H_1(L(p,q))=\mathbf{Z}_p$. Thus,every knot represents a torsion element in homology group. We give a construction of rational Seifert surface for arbitrary smooth knot when it has a regular projection on Heegaard torus of $L(p,q)$.
We assume that all knots mentioned in this note are smooth or piece-wise linear.

\section{Representation of a smooth knot in L(p,q)}
~~~~Let $V_{i}, (i=1, 2)$ be two solid torus $D^2\times S^1 $. Its meridian and longitude is denoted by$(\mu_i, \lambda_i)$. Then, in the sense of Heegaard decomposition, a lens space $L(p,q)$ can be described by $V_1\cup_\phi V_2$ where the gluing map $\phi:\partial V_2\rightarrow V_1$ ~is an orientation-reversing diffeomorphism given in standard longitude-meridian coordinates on the torus by the matrix
\\
$$
\begin{pmatrix}
-q&q'\\
p&-p'
\end{pmatrix}\in -SL_2(\mathbf{Z})$$\\
In particular, $\phi(\mu_2)=-q\mu_1+p\lambda_1$. This fact concludes that $H_1(L(p,q))=\langle \lambda_1 |~p\lambda_1=1 \rangle$. \\
~
\\~~~  Let $K$ be a knot in Lens space $L(p,q)$. Of course, after a small perturbation, it can be disjoint from the core $C_i=0\times S^1\subset D^2\times S^1 $of two solid torus at the same time. Please notice that $V_i\setminus C_i$ deformation retracts to its boundary$~\partial V_i$. Thus, the deformation retraction $P:L(p,q)\setminus V_1\cup V_2\rightarrow \partial V_1 $ projects $K$ onto Heegaard torus $\partial V_1$
\begin{definition}{(see chapter 3.E of \cite{RF})}\\
Assume K is a smooth knot. The deformation retraction $P$ is said to be \textbf{regular} for $K$ iff :\\
$\forall x \in \partial V_1, ~|P^{-1}(x)|=0,1,2$ and if  $2$, $P(K)$ intersects itself transversely at $x $
\end{definition}
\begin{remark}\label{remark1}
if P is not regular for K, then, after a small perturbation of K, P is regular. From now on, We assume K is in the interior of thickened torus $\partial V_1\times [-1,1]$ and the natural projection $\partial V_1\times [-1,1]\rightarrow \partial V_1$~is regular for K. We regard $L(p,q)$ is obtained from  $\partial V_1\times [-1,1]$ gluing $V_1$ to the lower boundary of this thickened torus and $V_2$ to the upper boundary.
\end{remark}
After above discussions, the reader can realize that such a knot K can be drawn on a fundamental domain of torus $\partial V_1$.Notice that $\partial V_1=T^2=\mathbf{R^2}/\mathbf{Z^2}$. The usual choice of fundamental domain of this torus is a square $[0,1]\times[0,1]\subset \mathbf{R^2}$. In this square, $[0,1]\times \{0\}$ represents $\mu_1$ while $\{0\} \times[0,1]$ represents $\lambda_1$
\begin{definition}{(see Def 2.1 of \cite{GD})}\\
The \textbf{twist toroidal diagram} of $\partial V_1\subset L(p,q)$ is a fundamental domain in $\mathbf{R^2}$ bounded by four straight line:\\
$$\begin{cases}
x=0\\
x=1\\
y=-\frac{q}{p} x\\
y=-\frac{q}{p}(x-1)
\end{cases}$$
\end{definition}
\begin{remark}
In twist toroidal diagram, it's also holds that $(0,1)(0,0)(1,0)$ represent a same point in $\partial V_1$.
The straight line $y=-\frac{q}{p}x $ has same direction as $\mu_2$.\end{remark}

\section{Construction of rational Seifert surface}
\subsection{Basic Idea}
~~~~By remark \ref{remark1}, we can draw $K$ on the twist toroidal diagram of $\partial V_1$. We want to find a "cobordism" surface (inside of $\partial V_1\times[-1,1]$) from $rK$ to a link $L'$ which is the union of several $(\pm\mu_2)-knot$ in $\partial V_1\times \{1\}$ and $(\pm\mu_1)-knot$ in $\partial V_1\times \{-1\}$. Then we attach several meridian discs of $V_i$ to this "cobordism", this so called "cobordism" should be a real rational Seifert surface of $K$. We will see later that $L'$ may contain several null-homologous component on the upper boundary of $\partial V_1\times[-1,1]$.
\subsection{Details of the construction}
The construction is divided into following steps:
\begin{enumerate}
    \item Replace crossings of $P(K)$ by short-cut arcs on the twist toroidal diagram. Or equivalently, cut the crossing point $A$ into two points $A_{0,1}$. Then, we get a torus link $L\subset\partial V_1\times\{0\}$ 
\begin{figure}[H]\label{step1}
    \centering
    \includegraphics[scale=0.28]{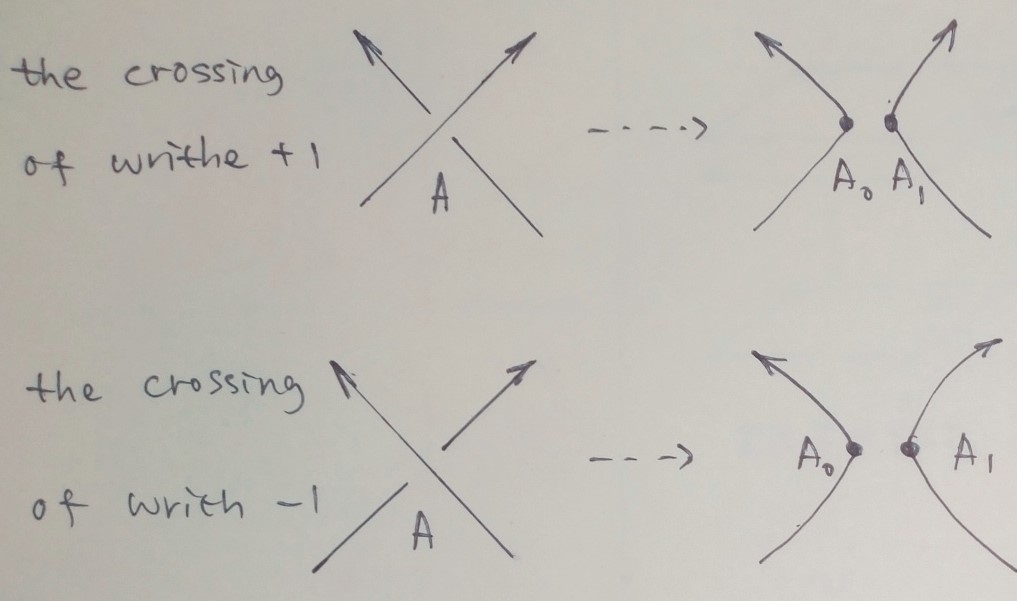}
    \caption{Make a crossing apart}
    
\end{figure}
    \item Computations:\\
    Compute $[K]=[L]\in H_1(\partial V_1)$ in coordinate $(\mu_1, \lambda_1)$. Assume that $[L]=n(a\mu_1+b\lambda_1)$ where $n,a,b\in \mathbf{Z}, g.c.d.(a,b)=1$. The coefficient $na(nb)$ and can be obtained by counting the algebraic intersection numbers of $L$ and $\lambda_1(\mu_1)$-curve.
    \\Also, Compute order $r$ of $[K]=[L]\in H_1(L(p,q))=\langle\lambda_1|p\lambda_1\rangle$. \\$$r=\frac{p}{g.c.d.(p, nb)}$$\\
    Then, $$r[L]
    =rna\mu_1+rnb\lambda_1
    =rna\mu_1+\frac{rnb}{p}(p\lambda_1)
    =rna\mu_1+\frac{rnb}{p}(q\mu_1+\mu_2)
    =(rna+\frac{rnbq}{p})\mu_1+\frac{rnb}{p}\mu_2$$
    \item Construct "cobordism" from link $L$ to $L'$ noticed above.
    \begin{enumerate}
        \item draw torus link $(rna+\frac{rnbq}{p})\mu_1$ on $\partial V_1\times\{-1\}$ (denoted by$L^-$)and $(-(rna+\frac{rnbq}{p})\mu_1)$ on $\partial V_1\times\{1\}$ s.t both torus link avoid a connected neighborhood of each crossing of $P(K)$ in the diagram where the crossing is now replaced by short-cut arcs.
        \begin{figure}[H]
            \centering
            \includegraphics[scale=0.25]{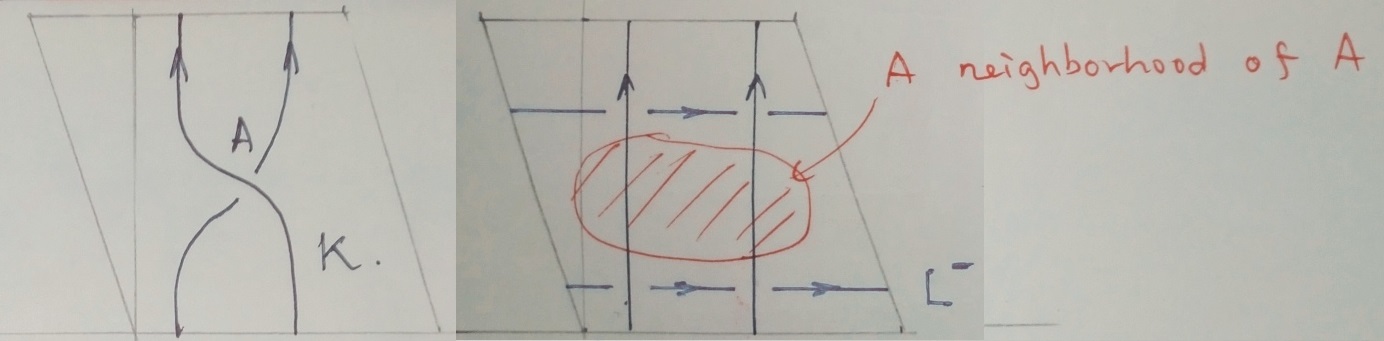}
            \caption{Here is a knot K in L(3,1), $[L]=2\lambda_1, r=3, r[L]=2\mu_1+2\mu_2$. The blue line $L^-$ a}
        \end{figure}
        For convenient, $(-(rna+\frac{rnbq}{p})\mu_1)$ on $\partial V_1\times\{1\}$ should be drawn a little bit above the $(rna+\frac{rnbq}{p})\mu_1$ on the diagram.
         \begin{figure}[H]
            \centering \includegraphics[scale=0.25]{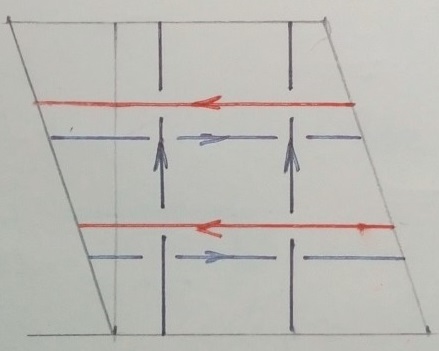}
            \caption{the red line of homotopy type $(-2\mu_1)$ is not far away from the blue.}
        \end{figure}
        \item draw torus link $rL$ on $\partial V_1\times\{1\}$. Here, $rL$ is r parallel copies of L. For convenience, one shouldn't draw $rL$ too far away from $L$.
        \begin{figure}[H]
            \centering \includegraphics[scale=0.25]{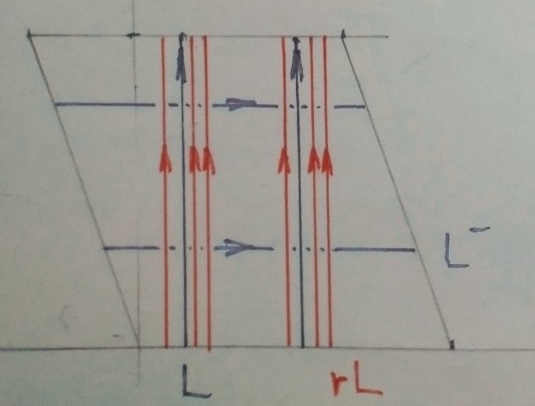}
            \caption{the red line $rL$ is far from L in the diagram we draw on.}
        \end{figure}
        \item At each intersection of $(-(rna+\frac{rnbq}{p})\mu_1)$ and $rL$ on $\partial V_1\times\{1\}$, replace intersection by smooth arc shown by the graph below.
        \begin{figure}[H]
            \centering \includegraphics[scale=0.3]{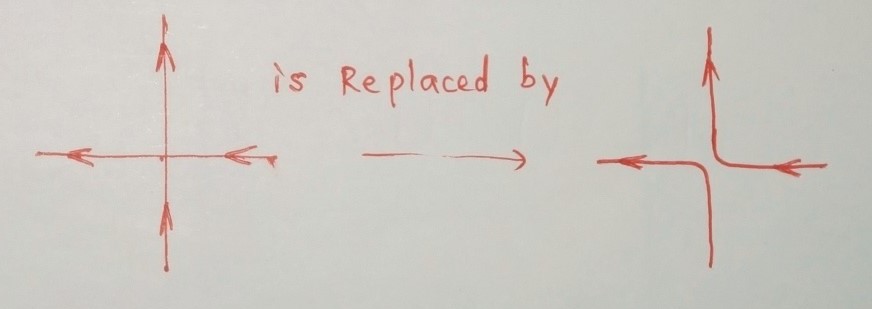}
            \caption{the other cases it quite similar.}
        \end{figure}
        Then, we get a link $L^+$ on $\partial V_1\times\{1\}$ with homology class $[L^+]=r[L]-(rna+\frac{rnbq}{p})\mu_1=\frac{rnb}{p}\mu_2 $. Therefore, its components is torus knot of $\pm\mu_2$ type or null-homologous (simple closed curve on torus). $L'$ is the union of $L^+$ and $L^-$ \begin{figure}[H]
            \centering \includegraphics[scale=0.3]{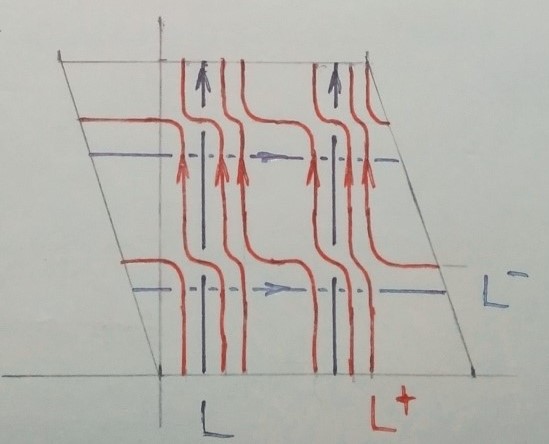}
            \caption{ the black is link $L$, the red is $L^+$ and the blue is $L^-$ }
        \end{figure}
        \item The "cobordism" of $L$ is actually bounded by $L$ and $L'$. Near the intersection of $L$ and $(rna+\frac{rnbq}{p})\mu_1$ link on the diagram, the "cobordism" is glued by the bands below. Outside the neighborhood, the "cobordism" is obtained by gluing r bands along $L$
        \begin{figure}[H]
            \centering \includegraphics[scale=0.28]{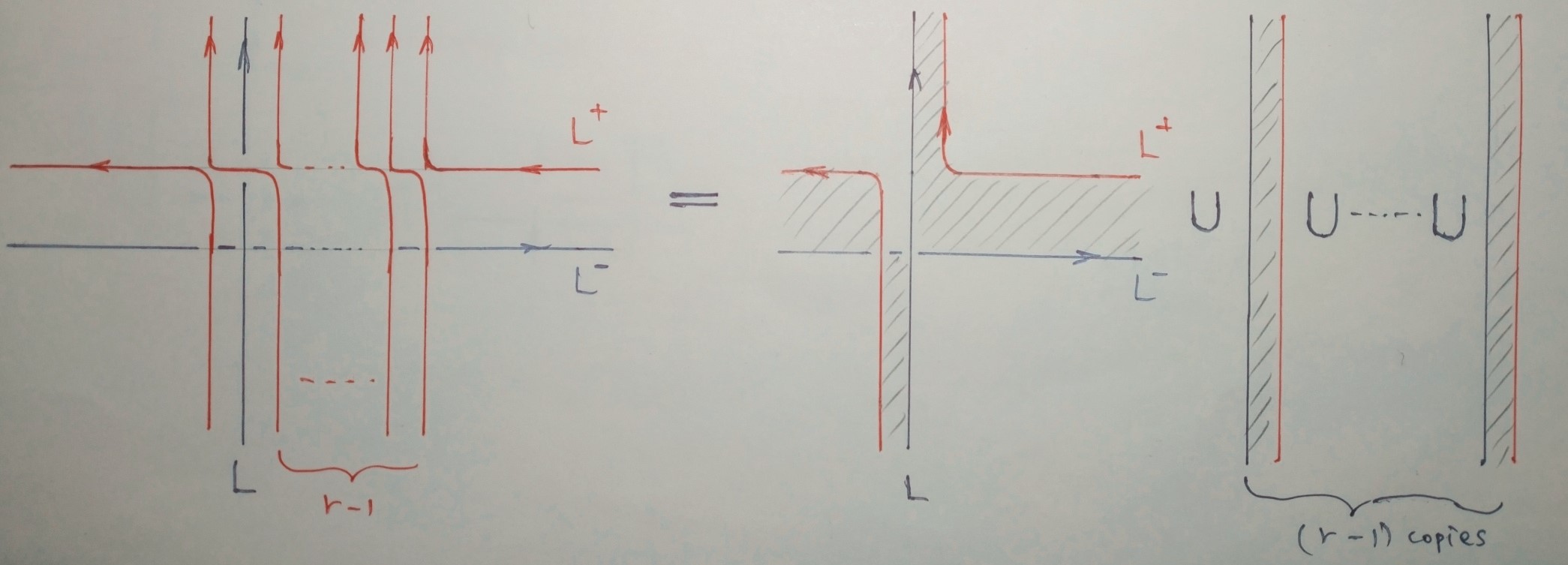}
            \caption{the other cases are quite similar with this figure}
        \end{figure}
        \item For a very special case when $[L]=0\in H_1(\partial V_1)$, $L'=\emptyset$ and $L$ consists of m($m\geq 0$) non-trivial torus knots of type $a\mu_1+b\lambda_1$, m torus knots of type $-(a\mu_1+b\lambda_1)$ and several null-homologous knots on torus. We construct disjoint m bands (i.e $S^1\times I$) and several discs bounded by null-homologous components of $L$ 
    \end{enumerate}
    \item Construct r-cover half-twist band as follow. Let $I\times I\times\{1,2,\dots ,r\}$ be k-copies of a square. Define equivalent relationship $\sim $ by: $(x,0,1)\sim (x,0,k)$ and $(x,1,1)\sim (x,1,k)$.
    \begin{figure}[H]
            \centering \includegraphics[scale=0.3]{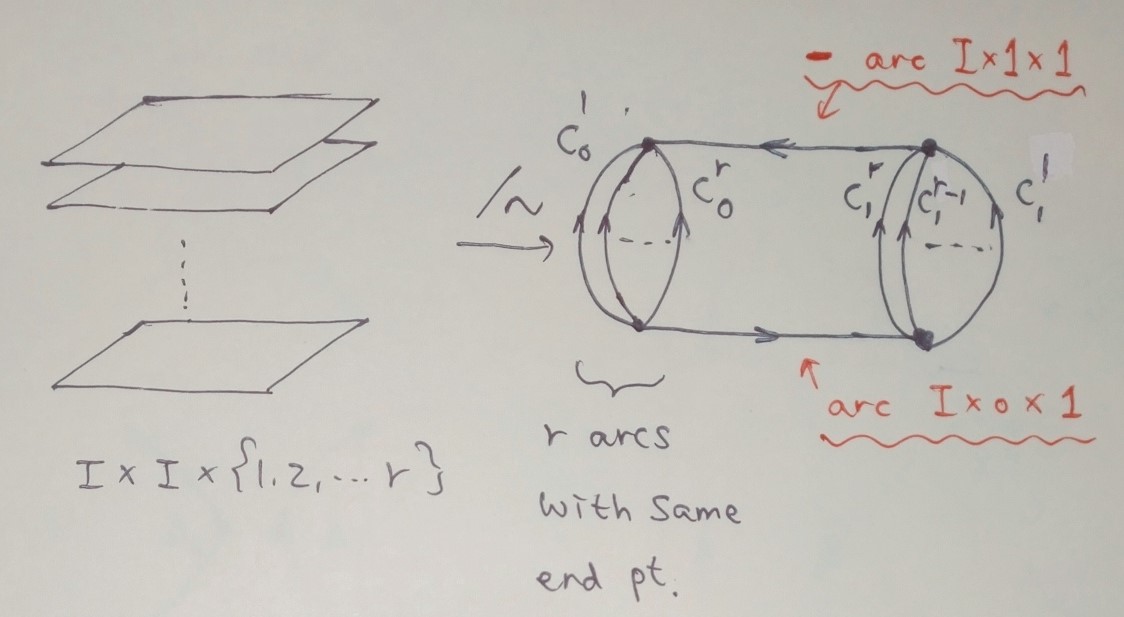}
            \caption{the other cases are quite similar with this figure}
        \end{figure}
        Then do a half-twist along straight line $I\times\{\frac{1}{2}\}\times\{0\}$ on the quotient space $I\times I\times\{1,2,\dots ,r\}/\sim$, the construction of r-cover half-twist band is done. Name arc $\{i\}\times I\times\{k\}$ by $c_i^k$ where $i=0,1; k=1,2,\dots,r$.\begin{figure}[H]
            \centering \includegraphics[scale=0.3]{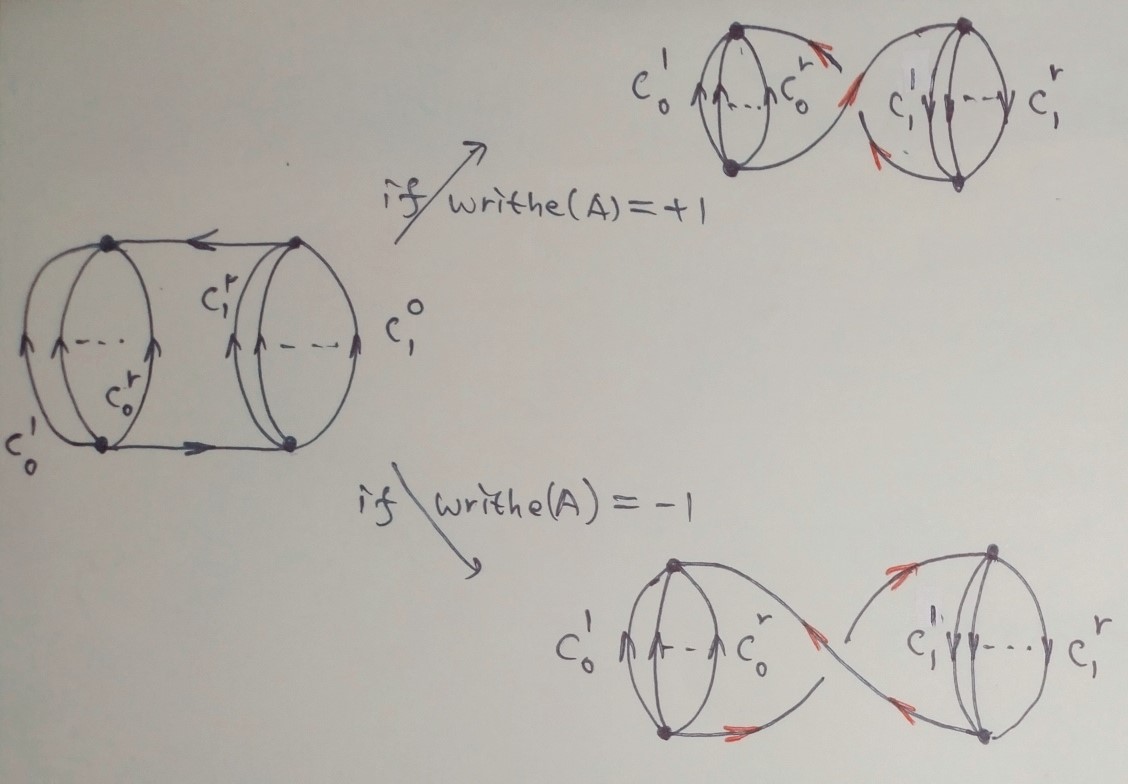}
            \caption{there are two type of r-cover half-twist band}
        \end{figure}
    \item In the first step, we cut apart the crossings (denoted by A) of $P(K)$ into two points $A_{0,1}$. \begin{figure}[H]
            \centering \includegraphics[scale=0.25]{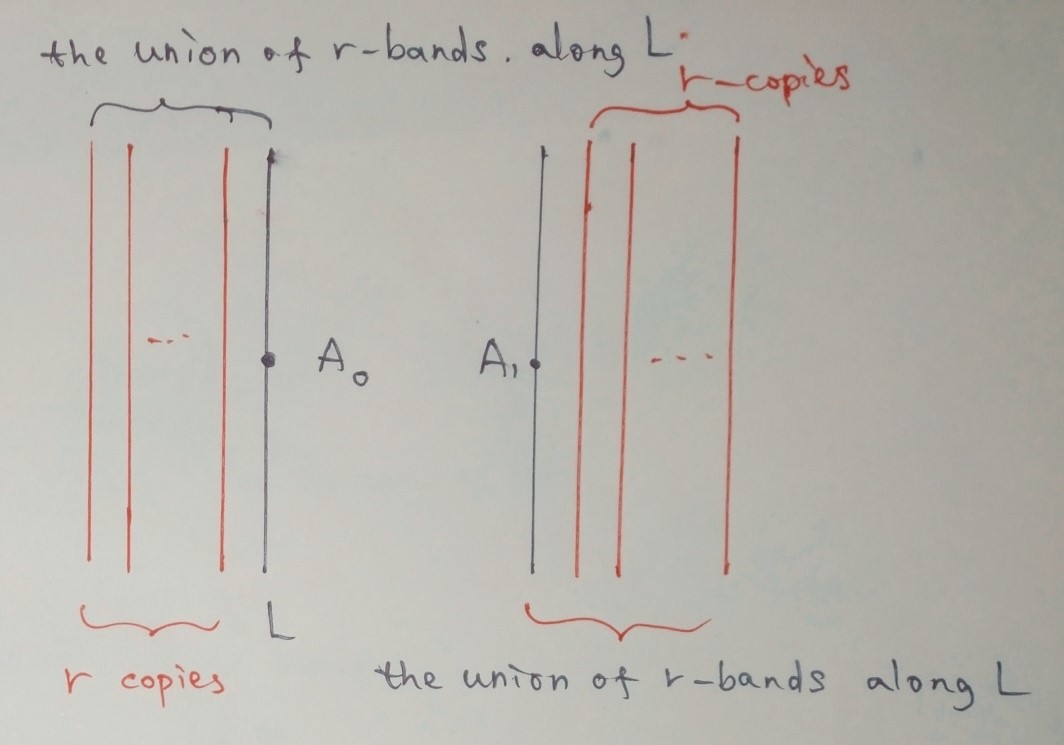}
            \caption{locally, the cobordism looked like above. Each local component is obtained by gluing r bands along L}
        \end{figure}Now we cut off a 3-ball $B_i$ of a very small radius centered at each $A_{i=0,1}$ from the "cobordism" constructed above. The boundary of 3-ball $\partial B_i$ intersects the cobordism at r arcs with same endpoints. These arcs is denoted by $\gamma_i^k$ where $i=0,1; k=1,2,\dots,r$. 
        \begin{figure}[H]
            \centering \includegraphics[scale=0.3]{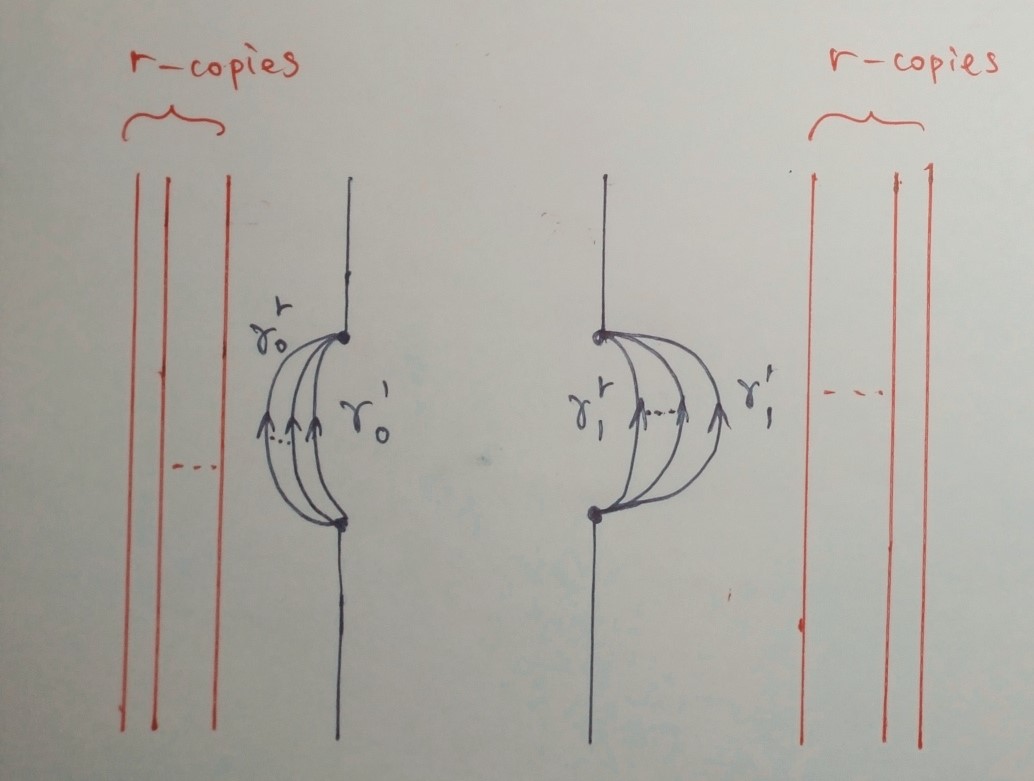}
            \caption{ $\gamma_i^k$ is marked in the figure}
        \end{figure}Now we attach r-cover half-twist band to the punctured cobordism described above by regarding $\gamma_0^k$ as $c_0^k$ and $\gamma_1^k$ as $-c_1^k$, $k=1,2,\dots,r$. One should take care that the type of r-cover half-twist band to be glued is depended on the writhe of this crossing. Then we get the cobordism from $rK$ to $L'$.
    \item Now we get the cobordism from $rK$ to $L'$. We gluing meridian discs of $V_1$ along $L^-$, and meridian discs of $V_2$ along the $\pm\mu_2$-type component of $L^+$.For those null-homologous component of $L^+$, we glue the discs bounded by them ,probably with a little push off the diagram s.t.the discs are disjoint.
\end{enumerate} 
Now we get a rational Seifert surface of $K$. It's not hard to compute its Eular characteristic. Also, we can find out how it wraps on K. See corollary below
\begin{corollary}
Let $K$ be a knot in the interior of $\partial V_1\times I$ with homotopy type $[K]=n(a\mu_1+b\lambda_1)$ where $n,a,b\in \mathbf{Z}, g.c.d.(a,b)=1$. Let $NK$ be a tubular neighborhood of $K$ with framing $(\mu_{NK},\lambda_{NK})$. Choose the longitude $\lambda_{NK}$ of $NK$ to be the one induced from the push-off of K along the positive direction of $I$. Then, the rational Seifert surface of $K$ intersects $\partial NK$ at a torus link with homology type: $$r\lambda_{NK}-(rn^2(a+\frac{bq}{p})b+r writhe(K))\mu_{NK}$$  where the writhe of K is the sum of index defined in the graph of the first step\ref{step1}.
\end{corollary}
\begin{proof}
the proof is not difficult noticing that the construction of cobordism of $L$ devotes $$-rn^2(a+\frac{bq}{p})b\mu_{NK}$$ and the attachment of r-cover half-twist bands devotes $$-rwrithe(K)\mu_{NK}$$.
\end{proof}
\section{Acknowledgement}
I would like to thank YouLin Li from SJTU. Without his help, I would not complete this thesis.

\end{document}